\documentclass[12pt,twoside]{amsart}
\usepackage{amsmath}
\usepackage{amssymb}
\usepackage{amscd}
\usepackage{xypic}
\usepackage{url}
\usepackage{graphicx}
\usepackage{hyperref}
\xyoption{all}
\makeatletter
 \def\@@and{y}
\makeatother

\title[Grupos magn\'eticos finitos]{Grupos magn\'eticos finitos y sus representaciones}

\author{Jos\'e Cantarero}
\thanks{}
\address{
\hfill\break Centro de Investigaci\'on en Matem\'aticas, A.C. Unidad M\'erida \\
\hfill\break Parque Cient\'ifico y Tecnol\'ogico de Yucat\'an \\ 
\hfill\break Carretera Sierra Papacal--Chuburn\'a Puerto Km 5.5 \\
\hfill\break Sierra Papacal, M\'erida, YUC 97302 \\
\hfill\break Mexico.
\hfill\break {\emph{Email address: }}{\tt cantarero@cimat.mx}}

\author{Higinio Serrano Garcia}
\thanks{}
\address{
\hfill\break Institute of Mathematics and Informatics \\
\hfill\break Bulgarian Academy of Sciences \\ 
\hfill\break Acad. Georgi Bonchev Str., Block 8 \\
\hfill\break 1113 Sofia \\
\hfill\break Bulgaria.
\hfill\break {\emph{Email address: }}{\tt hserrano@math.cinvestav.mx}}

\theoremstyle{plain}
\newtheorem{theorem}{Teorema}[section]
\newtheorem{proposition}[theorem]{Proposici\'on}
\newtheorem{corollary}[theorem]{Corolario}
\newtheorem{lemma}[theorem]{Lema}

\theoremstyle{definition}
\newtheorem{definition}[theorem]{Definici\'on}
\newtheorem{remark}[theorem]{Observaci\'on}
\newtheorem{example}[theorem]{Ejemplo}



\begin{document}

\begin{abstract}
En este art\'iculo damos una introducci\'on elemental a la teor\'ia de representaciones de grupos magn\'eticos
finitos desde un punto de vista puramente matem\'atico.
\end{abstract}

\maketitle

\section{Introducci\'on}

Una de las \'areas m\'as grandes e importantes dentro del campo de la f\'isica es la llamada \emph{materia condensada y f\'isica del estado s\'olido}, que a grandes rasgos estudia sistemas mec\'anico-cu\'anticos con una gran cantidad de part\'iculas que interact\'uan entre s\'i. La alta complejidad de los sistemas que esta disciplina estudia se puede reducir gradualmente al usar una poderosa herramienta: la \emph{simetr\'ia}. Dentro del cat\'alogo de las simetr\'ias, que aparecen en el contexto de la materia condensada, est\'an los \emph{grupos cristalogr\'aficos magn\'eticos}; grupos que contienen traslaciones, rotaciones, reflexiones e inversiones temporales. Estos grupos han sido ampliamente estudiados y clasificados en su totalidad desde finales del siglo XIX. Evgraf Fedorov \cite{fedorov1971symmetry} realiz\'o la primera clasificaci\'on de grupos cristalogr\'aficos que no contienen inversi\'on temporal. Posteriormente, Shubnikov y Belov \cite{ShubnikovBelov1964} incorporaron las inversiones temporales. Para una descripci\'on m\'as actual se puede consultar \cite{bradley2010mathematical} o \cite{Litvin2013}.

Estudiar los grupos cristalogr\'aficos magn\'eticos y sus representaciones ha sido parte fundamental en la predicci\'on de nuevas fases topol\'ogicas de la materia tales como aislantes topol\'ogicos \cite{RevModPhys.82.3045}, materiales altermagn\'eticos \cite{PhysRevB.111.085127}, aislantes de Chern \cite{cherninsulators}, etc. La principal diferencia con los grupos usuales y sus representaciones es que las representaciones de grupos magn\'eticos incorporan operadores tanto lineales como antilineales. Por ello es recomendable, m\'as no indispensable, que el lector haya tenido un contacto con la teor\'ia de representaciones de grupos finitos \cite{FultonHarris1991} pues as\'i notar\'a las similitudes y diferencias de las teor\'ias.

La intenci\'on de este art\'iculo es dar una introducci\'on moderna a la teor\'ia de representaciones de grupos magn\'eticos finitos, iniciada por E. Wigner \cite{wigner}, poniendo especial \'enfasis en resultados claves como el lema de Schur (lema \ref{lemadeschur}) y la clasificaci\'on de representaciones irreducibles (teorema \ref{claswigner}). El enfoque es completamente matem\'atico, y las conexiones e interpretaciones f\'isicas mencionadas en el texto no son necesarias para comprender el contenido. Con el objetivo de hacer este tema m\'as accesible, el escrito se formula principalmente en t\'erminos de \'algebra lineal y teor\'ia de grupos, con breves menciones a conceptos categ\'oricos que, si bien no son esenciales, pueden resultar \'utiles para el lector familiarizado con ellos.
\vspace{0.25cm}

\noindent \begin{bf}Notaci\'on\end{bf}: Denotaremos a la clase del entero $k$ en $\mathbb{Z}/n$ simplemente como $k$. Tambi\'en usaremos $\mathrm{id}_X$ para denotar la funci\'on identidad $X \to X$ o solamente $\mathrm{id}$ cuando $X$ se pueda deducir
del contexto.
\vspace{0.25cm}

\noindent \begin{bf}Agradecimientos\end{bf}: HS agradece el apoyo de la Simons Foundation, apoyo SFI-MPS-T-Institutes-00007697, y al Ministerio de Educaci\'on y Ciencia de la Rep\'ublica de Bulgaria, apoyo DO1-239/10.12.2024. La elaboraci\'on de este art\'iculo fue apoyada por CONAHCYT (ahora Secihti) en el a\~no 2023 mediante el proyecto de Ciencia de Frontera CF-2023-I-2649.

\section{Grupos magn\'eticos}

Consideremos la extensi\'on de campos
\begin{equation*}
    \xymatrix{\mathbb{C} \ar@{-}[d] \\ \mathbb{R} }   
\end{equation*}
y denotemos al grupo de automorfismos de dicha extensi\'on por 
\begin{equation*}
    \mathrm{Aut}_{\mathbb{R}}(\mathbb{C})=\left\{f \colon \mathbb{C}\longrightarrow\mathbb{C}\,\Bigg|\, \begin{array}{c}
         f \text{ es un isomorfismo de campos,} \\ f(x)=x \text{ para todo } x\in \mathbb{R}
    \end{array} \right\}.    
\end{equation*}
Este grupo resulta ser isomorfo a $\mathbb{Z}/2$ como veremos a continuaci\'on.

\begin{proposition}
    El grupo de automorfismos de la extensi\'on de campos $\mathbb{C}/\mathbb{R}$ consta de dos elementos, concretamente
    \begin{equation*}
        \mathrm{Aut}_{\mathbb{R}}(\mathbb{C})=\{\mathrm{id}, \mathbb{K}\},
    \end{equation*}
    donde $\mathbb{K}$ es el automorfismo definido por la conjugaci\'on compleja. En particular, $\mathrm{Aut}_{\mathbb{R}}(\mathbb{C})$ es isomorfo a $\mathbb{Z}/2$.
\end{proposition}
\begin{proof}
    Dado $f\in \mathrm{Aut}_{\mathbb{R}}(\mathbb{C})$ y $z=x+iy$, se tiene 
    \begin{align*}
        f(z) & = f(x+iy) \\ &= f(x)+f(iy) \\ & = x+f(i)y.
    \end{align*}
    Vemos que $f$ queda totalmente determinado por su valor en $i$. Este valor lo podemos calcular de la igualdad
    \begin{align*}
        f(i)^2=f(i^2)=f(-1)=-1,
    \end{align*}
    de donde se sigue que $f(i)=\pm i$. Si $f(i)=i$, entonces $f$ es la identidad y si $f(i)=-i$, entonces $f=\mathbb{K}$. Por \'ultimo, notemos que conjugar dos veces produce la identidad, es decir, $\mathbb{K}^2=\mathrm{id}$.
\end{proof}

El grupo de automorfismos de la extensi\'on $\mathbb{C}/\mathbb{R}$ nos ayudar\'a a codificar la forma en la que las simetr\'ias magn\'eticas actuar\'an en los espacios vectoriales. Algunos elementos del grupo ser\'an operadores lineales mientras que otros ser\'an antilineales. 

\begin{definition}
Un \textbf{grupo magn\'etico} es una pareja $(G,\phi)$ donde $G$ es un grupo y $\phi \colon G \longrightarrow \mathbb{Z}/2$ es un epimorfismo de grupos.
\end{definition}
Al n\'ucleo del homomorfismo $\phi$ lo denotaremos por $G_0$. Esta informaci\'on la resumiremos con la siguiente sucesi\'on exacta:
\[ 1 \longrightarrow G_0 \longrightarrow G \stackrel{\phi}{\longrightarrow} \mathbb{Z}/2 \longrightarrow 1, \]
la cual tambi\'en se conoce como una extensi\'on de grupos. En general, cualquier epimorfismo $f \colon G \to Q$ de grupos define una extensi\'on de grupos
\[ 1 \longrightarrow \text{Ker}(f) \longrightarrow G \stackrel{f}{\longrightarrow} Q \longrightarrow 1. \]
donde el morfismo $\text{Ker}(f)\to G$ es la inclusi\'on. Tambi\'en es posible construir una extensi\'on de grupos a partir de dos grupos $N$, $Q$, y un homomorfismo $h \colon Q \to \text{Aut}(N)$. El producto semidirecto $ N \rtimes_h Q$ es el producto cartesiano $N \times Q$ equipado con el producto
\[ (n,q) (m,p) = (n \cdot h(q)(m),q \cdot p), \]
que lo convierte en un grupo. Es com\'un referirse a $h(q)(m)$ como la acci\'on de $q$ sobre $m$ pues define una acci\'on de $Q$ sobre $N$. Existe un epimorfismo $\pi \colon N \rtimes_h Q \to Q$ que env\'ia $(n,q)$ a $q$, cuyo n\'ucleo es isomorfo a $N$, por lo que da lugar a una sucesi\'on exacta
\[ 1 \longrightarrow N \longrightarrow N \rtimes_h Q \stackrel{\pi}{\longrightarrow} Q \longrightarrow 1. \]
Cuando $h$ es claro a partir del contexto, se denota simplemente por $N \rtimes Q$.

\begin{example}
    Un ejemplo de estos grupos surge naturalmente de considerar las simetr\'ias cristalogr\'aficas. Intuitivamente, un cristal es un s\'olido cuyos \'atomos est\'an dispuestos en varias ret\'iculas tridimensionales. Un cristal con esp\'in es un cristal ordinario en el que cada \'atomo lleva un peque\~no “im\'an” interno, llamado esp\'in, que no est\'a determinado \'unicamente por su posici\'on y cuya organizaci\'on colectiva da lugar a propiedades magn\'eticas y otras propiedades f\'isicas del material. Las simetr\'ias de los cristales con esp\'in viven dentro del producto cartesiano de $\mathbb{Z}/2$ con el grupo de transformaciones afines\footnote{Una transformaci\'on af\'in de $\mathbb{R}^3$ es una funci\'on $f \colon \mathbb{R}^3\longrightarrow \mathbb{R}^3$ que preserva la distancia euclidiana.} de $\mathbb{R}^3$:
\[\operatorname{Aff}(\mathbb{R}^3)\times \mathbb{Z}/2\cong\left(\mathbb{R}^3\rtimes O(3)\right)\times \mathbb{Z}/2,\]
donde el grupo $\mathbb{R}^3$ contiene las simetr\'ias de traslaci\'on del cristal, el grupo $O(3)$ contiene las simetr\'ias ortogonales, es decir, rotaciones compuestas con inversiones espaciales, y el grupo $\mathbb{Z}/2$ contiene a la simetr\'ia de inversi\'on temporal. Para ser m\'as preciso, las simetr\'ias de un cristal esp\'in est\'an codificadas por una subextensi\'on como se muestra a continuaci\'on:
\[\xymatrix{1 \ar[r] &\Lambda \ar@{^(->}[]-<0pt,8pt>;[d]  \mkern-7mu \ar[r] & G \ar[r] \ar@{^(->}[]-<0pt,8pt>;[d] \mkern-7mu & P \ar[r] \ar@{^(->}[]-<0pt,8pt>;[d] \mkern-7mu & 1 \\ 1\ar[r]+<-5pt,0pt> & \mkern+9mu \mathbb{R}^3 \ar[r] & \left(\mathbb{R}^3\rtimes O(3)\right)\times \mathbb{Z}/2\ar[r] & O(3)\times \mathbb{Z}/2 \ar[r] & 1, }\]
donde $\Lambda$ es isomorfo a $\mathbb{Z}^3$ y $P$ es un grupo finito. De manera m\'as expl\'icita, es un subgrupo $G$ de $\operatorname{Aff}(\mathbb{R}^3) \times \mathbb{Z}/2$ tal que su intersecci\'on $\Lambda$ con el subgrupo $\mathbb{R}^3$ de traslaciones es isomorfo a $\mathbb{Z}^3$, y tal que $G/\Lambda$ es finito. Al grupo $\Lambda$ se le conoce como \textbf{ret\'icula de Bravais}, al grupo $G$ como el \textbf{grupo cristalogr\'afico magn\'etico} y a $P$ como el \textbf{grupo puntual magn\'etico}.

Notemos que tanto $G$ como $P$ est\'an equipados con un homomorfismo a $\mathbb{Z}/2$. En el caso de $P$, est\'a dado por la inclusi\'on a $O(3)\times \mathbb{Z}/2$ seguido de la proyecci\'on al segundo factor, y para $G$, componemos el epimorfismo $G \to P$ con este homomorfismo $P \to \mathbb{Z}/2$ de manera que obtenemos un diagrama conmutativo:
\[\xymatrix{G \ar[r] \ar[rd]_{\phi_G} & P \ar[d]^{\phi_P}\\ & \mathbb{Z}/2.}\]

\begin{figure}[h]
\caption{Vista superior del compuesto $U_2Pd_2In$, extra\'ida de \cite{Gallego:ks5532}. Las esferas con flechas representan \'atomos de Uranio, las grises son \'atomos de Paladio y las rosas son \'atomos de Indio.}
\centering
\includegraphics[width=0.5\textwidth]{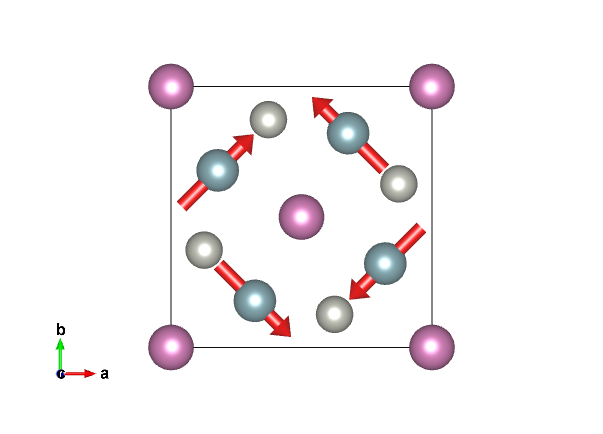}\label{mol}
\end{figure}

En la figura \ref{mol} se puede apreciar al compuesto $U_2Pd_2In$. Esta mol\'ecula tiene como grupo puntual magn\'etico $(\mathbb{Z}/4,\operatorname{mod}2)$ donde el generador de $\mathbb{Z}/4$ es una rotaci\'on de noventa grados en el plano $a-b$ seguida del operador de inversi\'on temporal. Las flechas rojas denotan la direcci\'on del esp\'in. Notemos que la inversi\'on temporal solo afecta la direcci\'on del esp\'in.  
El lector que quiera conocer un poco m\'as acerca de la f\'isica detr\'as de estos grupos puede consultar \cite{Yang2024} y todas las referencias que ah\'i se encuentran, as\'i como \cite{bradley2010mathematical}, que contiene la clasificaci\'on completa.
\end{example}

\begin{example}
    Sea $V$ un espacio vectorial sobre $\mathbb{C}$ de dimensi\'on finita. El \textbf{grupo magn\'etico general lineal de $V$}
 es
 \[\textbf{GL}(V)=\left\{A \colon V\rightarrow V \, \Bigg|\, \begin{array}{c}
    A \text{ es un operador invertible} \\
     \text{lineal o antilineal  } 
 \end{array} \right\}.\]

Un operador $A \colon V\longrightarrow V$ entre espacios vectoriales complejos es lineal si $A(zv+w)=zA(v)+A(w)$ y antilineal si $A(zv+w)=\overline{z}A(v)+A(w)$ para $ v,w\in V$ y $z\in\mathbb{C}$. Es claro que $\textbf{GL}(V)$ es un grupo y adem\'as viene provisto de manera natural con el homomorfismo
\begin{align*}
    \phi \colon \textbf{GL}(V) & \longrightarrow \mathbb{Z}/2, \\ A& \longmapsto \left\{\begin{array}{cl}
        1, & \text{si } A \text{ es antilineal}, \\
        0, & \text{si } A \text{ es lineal.} 
    \end{array}\right.
\end{align*}
El n\'ucleo de este homomorfismo es el grupo $\operatorname{GL}(V)$ de operadores lineales invertibles de $V$. Si $V=\mathbb{C}^n$, 
\[\text{?`podemos describir al grupo } \textbf{GL}(V) \text{ en t\'erminos de matrices?}\] Para responder a la pregunta, primero notemos que tenemos dos funciones
\begin{align}\label{isodegrupmag}
    \Psi \colon \textbf{GL}(\mathbb{C}^n)&\longrightarrow \operatorname{GL}(n)\times \mathbb{Z}/2, \\
    \nonumber A&\longmapsto (A\circ\mathbb{K}^{\phi(A)}, \phi(A)),
\end{align}

\begin{align*}
    \Phi \colon \operatorname{GL}(n)\times \mathbb{Z}/2 & \longrightarrow \textbf{GL}(\mathbb{C}^n),
    \\ \nonumber (A,\varepsilon) & \longmapsto A\circ \mathbb{K}^{\varepsilon},
\end{align*}
donde $\mathbb{K}^0 = \mathrm{id} \colon \mathbb{C}^n\longrightarrow \mathbb{C}^n$ y  $\mathbb{K}^1 = \mathbb{K} \colon \mathbb{C}^n\longrightarrow \mathbb{C}^n$ denota el operador de conjugaci\'on compleja entrada a entrada. Invitamos al lector a comprobar que $\Phi$ es la inversa de $\Psi$. Sin embargo, $\Psi$ y $\Phi$ no son homomorfismos puesto que
\begin{align*}
    \Psi(AB)&=(AB\circ \mathbb{K}^{\phi(AB)},\phi(AB)) \\ &= (AB\circ \mathbb{K}^{\phi(A)}\mathbb{K}^{\phi(B)},\phi(A)\phi(B))
\end{align*}
y
\begin{align}\label{semiprod}
    \nonumber\Psi(A)\Psi(B)&=(A\circ \mathbb{K}^{\phi(A)},\phi(A))(B\circ\mathbb{K}^{\phi(B)},\phi(B)) \\ \nonumber &=  (A\circ \mathbb{K}^{\phi(A)}\circ B\circ \mathbb{K}^{\phi(B)},\phi(A)\phi(B)) \\ & = (A\mathbb{K}^{\phi(A)}(B)\circ\mathbb{K}^{\phi(A)}\mathbb{K}^{\phi(B)},\phi(A)\phi(B))
\end{align}
no son iguales. No obstante, la ecuaci\'on \ref{semiprod} nos dice que si consideramos el producto semidirecto $\operatorname{GL}(n)\rtimes \mathbb{Z}/2$ donde la acci\'on de $\mathbb{Z}/2$ sobre $\operatorname{GL}(n)$ conjuga complejamente todas las entradas de la matriz, entonces $\Psi$ define un isomorfismo
\begin{equation*}
    \textbf{GL}(\mathbb{C}^n) \overset{\cong}{\longrightarrow} \text{GL}(n)\rtimes \mathbb{Z}/2.
\end{equation*}
\end{example}

\begin{definition}
    Un \textbf{morfismo} entre dos grupos magn\'eticos $(G,\phi_G)$ y $(H,\phi_H)$ es un homomorfismo de grupos $f \colon G\longrightarrow H$ tal que el siguiente diagrama conmuta
    \[\xymatrix{G \ar[rr]^{f} \ar[dr]_{\phi_G}&& H \ar[dl]^{\phi_H} \\ &\mathbb{Z}/2.&}\]
\end{definition}

Denotemos por $\textbf{Grp}$ la \textbf{categor\'ia de grupos magn\'eticos} y por $\operatorname{Grp}$ la \textbf{categor\'ia de grupos}. 

\section{Representaciones}

Ahora estamos listos para definir las representaciones de los grupos magn\'eticos. Los resultados que mostramos en esta secci\'on se basan en los trabajos de E. Wigner (secci\'on 26 de \cite{wigner}), en donde se enfatiza la importancia as\'i como el uso de los grupos y sus representaciones en el estudio de la mec\'anica cu\'antica. Es deseable, aunque no necesario, que el lector ya haya tenido contacto con la teor\'ia de representaciones de grupos finitos. Recomendamos revisar los primeros tres cap\'itulos del cl\'asico libro \cite{FultonHarris1991}, as\'i se podr\'an apreciar las similitudes y diferencias con respecto a lo que aqu\'i mostraremos. De ahora en adelante el grupo $G$ ser\'a finito, a menos que se diga lo contrario.

En general, la teor\'ia de representaciones busca representar objetos matem\'aticos abstractos mediante transformaciones de estructuras matem\'aticas m\'as concretas, con el prop\'osito de reducir la abstracci\'on o facilitar c\'alculos. Por ejemplo, una representaci\'on de un grupo $G$ en un espacio vectorial complejo $V$ asigna a cada elemento del grupo un isomorfismo lineal de $V$ mediante un homomorfismo $G \to \operatorname{GL}(V)$. 

\begin{definition}
    Una \textbf{representaci\'on} de un grupo magn\'etico $(G,\alpha)$ en un espacio vectorial complejo $V$ de dimensi\'on finita es un homomorfismo de grupos magn\'eticos
    \[\rho \colon (G,\alpha)\longrightarrow (\textbf{GL}(V),\phi).\]
\end{definition}

Este homomorfismo determina una acci\'on de $G$ sobre $V$ dada por $g \cdot v = \rho(g)(v)$. La compatibilidad del homomorfismo $\rho$ con los homomorfismos a $\mathbb{Z}/2$ significa que los elementos de $G$ en el n\'ucleo de $\phi$ act\'uan sobre $V$ de forma lineal mientras los que no est\'an en el n\'ucleo lo har\'an de forma antilineal. Si $V=\mathbb{C}^n$, entonces por el isomorfismo dado en la ecuaci\'on \ref{isodegrupmag} podemos escribir cualquier representaci\'on como un homomorfismo
\begin{align*}
    \rho \colon G& \longrightarrow \operatorname{GL}(n)\rtimes \mathbb{Z}/2, \\ g& \longmapsto (\rho'(g),\phi(g)),
\end{align*}
para cierta funci\'on $\rho' \colon G\longrightarrow \operatorname{GL}(n)$. Usando que $\rho$ es un homomorfismo y la forma de la multiplicaci\'on en el producto semidirecto, vemos que es $\rho'$ no es un homomorfismo pero satisface la condici\'on similar
\begin{align}\label{matriz}
    \rho'(gh)&=\rho'(g)\circ \mathbb{K}^{\phi(g)}(\rho'(h)). 
\end{align}

\begin{example}
    Consideremos el grupo magn\'etico $(\mathbb{Z}/4,\operatorname{mod}2)$. Dos de sus representaciones est\'an dadas por
    \begin{align*}
        \mathbb{Z}/4& \longrightarrow \textbf{GL}(\mathbb{C}), \\
        1&\longmapsto \mathbb{K},
    \end{align*}

    \begin{align*}
        \mathbb{Z}/4& \longrightarrow \textbf{GL}(\mathbb{C}^2), \\
        1&\longmapsto \begin{pmatrix}
            0&-1\\1&0
        \end{pmatrix}\circ\mathbb{K}.
    \end{align*}
    En general, cualquier representaci\'on de este grupo es de la forma 
    \begin{equation*}
        1\longmapsto A\circ\mathbb{K}    
    \end{equation*}
    con $A\in \operatorname{GL}(n)$, y adem\'as satisface que
    \begin{align*}
        \mathrm{id} & = (A\mathbb{K})^4 \\ 
        & = A\circ\mathbb{K}\circ A\circ\mathbb{K}\circ A\circ\mathbb{K}\circ A\circ\mathbb{K}\\ & = A\overline{A}\circ\mathbb{K}^2\circ A\overline{A}\circ\mathbb{K}^2 \\ & = (A\overline{A})^2.
    \end{align*}
 El mismo argumento muestra que cualquier representaci\'on del grupo magn\'etico $(\mathbb{Z}/2n,\operatorname{mod}2)$, est\'a dada por
    \begin{align}
        \nonumber \mathbb{Z}/2n & \longrightarrow \textbf{GL}(\mathbb{C}^m), \\ \nonumber 1 & \longmapsto A\circ \mathbb{K},
    \end{align}
    donde $A\in \operatorname{GL}(m)$ y $(A\overline{A})^n=\mathrm{id}$.
\end{example}

\begin{example}
Para cualquier grupo magn\'etico finito $(G,\phi)$, su representaci\'on regular $\mathbb{C}[G,\phi]$ es el espacio vectorial complejo $\mathbb{C}[G]$ cuya base es $G$, con la acci\'on de $(G,\phi)$ dada por 
    \[g\cdot \sum_i \lambda_i g_i= \sum_i \mathbb{K}^{\phi(g)}(\lambda_i)(gg_i).\]
\end{example}

\begin{definition}
    Un \textbf{morfismo} entre dos representaciones $\rho_i \colon G \longrightarrow \textbf{GL}(V_i)$ $i=1,2$ es una transformaci\'on lineal $T \colon V_1\longrightarrow V_2$ tal que el diagrama  
    \[\xymatrix{V_1 \ar[r]^{\rho_1(g)} \ar[d]_{T} & V_1 \ar[d]^{T}\\ V_2 \ar[r]_{\rho_2(g)} & V_2  
    }\]
    conmuta para todo $g\in G$.
\end{definition}

Denotemos la categor\'ia de representaciones del grupo magn\'etico $(G,\phi)$ por $\textbf{Rep}(G,\phi)$ y por $\operatorname{Rep}(G_0)$ a la categor\'ia de representaciones
(cl\'asicas) del grupo $G_0$. Al conjunto de morfismos entre dos representaciones $\rho_1$ y $\rho_2$ de $(G,\phi)$ \'o $G_0$ lo denotaremos por $\operatorname{Hom}_{(G,\phi)}(\rho_1,\rho_2)$ \'o $\operatorname{Hom}_{G_0}(\rho_1,\rho_2)$, respectivamente.

\section{Representaciones inducidas}

Notemos que al restringir una representaci\'on de $(G,\phi)$ al subgrupo $G_0$ obtenemos una representaci\'on de $G_0$. De hecho, restricci\'on define un funtor
\begin{equation*}
    \text{Res}_{G_0}^G \colon \textbf{Rep}(G,\phi) \longrightarrow \operatorname{Rep}(G_0).
\end{equation*}
Si ahora comenzamos con una representaci\'on del grupo $G_0$, entonces \[\text{?`c\'omo podemos asociarle una representaci\'on de $(G,\phi)$?}\]
\begin{definition}
    Sea $\rho \colon G_0\longrightarrow \operatorname{GL}(V)$ una representaci\'on del grupo $G_0$. La \textbf{representaci\'on inducida} de $V$ a $(G,\phi)$ est\'a dada por el espacio vectorial
    \[\operatorname{Ind}^{G}_{G_0}V= \mathbb{C}[G,\phi]\otimes_{\mathbb{C}[G_0]} V, \]
    donde $\mathbb{C}[G,\phi]$ es la representaci\'on regular del grupo magn\'etico $(G,\phi)$ y $\mathbb{C}[G_0]$ es la representaci\'on regular del grupo $G_0$. La acci\'on del grupo magn\'etico $(G,\phi)$ en $\operatorname{Ind}^{G}_{G_0}V$ est\'a dada por la acci\'on en el primer factor. 
\end{definition}
La construcci\'on anterior se puede extender a los morfismos de representaciones para producir un funtor
\[\operatorname{Ind}^{G}_{G_0} \colon \operatorname{Rep}(G_0) \longrightarrow \textbf{Rep}(G,\phi).\]
Los funtores $\operatorname{Res}_{G_0}^G$ y $\operatorname{Ind}^{G}_{G_0}$ son muy importantes, pues, entre otras cosas, nos permitir\'an construir representaciones irreducibles de $(G,\phi)$ a partir de representaciones irreducibles de $G_0$. Pero antes daremos otra forma m\'as elemental de pensar la inducci\'on.

Sea $\rho \colon G_0\longrightarrow \operatorname{GL}(V)$ una representaci\'on de $G_0$ y fijemos un elemento $a\in G\backslash G_0$. Consideremos el espacio vectorial
\[ aV =\{av\,|\, v\in V\} \]
con las operaciones
\[ av+aw=a(v+w), \qquad \qquad \lambda\cdot av=a(\overline{\lambda}v). \]
Definimos la \textbf{representaci\'on conjugada} de $\rho$ como
\begin{align}\label{repconjugda}
    \rho^* \colon G_0 & \longrightarrow \operatorname{GL}(aV), \\ \nonumber g& \ \longmapsto \overline{\rho(a^{-1}ga)}, 
\end{align}
donde
\[ \overline{\rho(a^{-1}ga)}(av) = a[\rho(a^{-1}ga)(v)]. \]
Veamos que $\rho^*$ es un homomorfismo de grupos
\begin{align*}
    \rho^*(gh)(av) &= a[\rho(a^{-1}gha)(v)] \\ & = a[\rho(a^{-1}gaa^{-1}ha)(v)] \\&= a[\rho(a^{-1}ga) \rho(a^{-1}ha)(v)] \\&= \rho^*(g)(a[\rho(a^{-1}ha)(v)]) \\&= [\rho^*(g) \circ \rho^*(h)](av).
\end{align*}
Por lo tanto $\rho^*$ es una representaci\'on de $G_0$. Es conveniente explicar la notaci\'on en detalle, para lo cual introducimos los conjugados de operadores lineales.

\begin{definition}
Dado un operador lineal $S \colon \mathbb{C}^n \to \mathbb{C}^n$, su conjugado es el operador lineal $\overline{S} = \mathbb{K} S \mathbb{K}$.
\end{definition}

Cuando $V=\mathbb{C}^n$, el isomorfismo lineal $a\mathbb{C}^n \to \mathbb{C}^n$ dado por $av \mapsto \overline{v}$ se convierte en un isomorfismo de representaciones si a $\mathbb{C}^n$ le damos la acci\'on
\[ g w = \overline{\rho(a^{-1}ga)(\bar{w})} = \mathbb{K} \rho(a^{-1}ga) \mathbb{K}(w). \]
Si identificamos $a\mathbb{C}^n$ con $\mathbb{C}^n$ mediante este isomorfismo, obtenemos
\begin{equation}
\label{conjugadafacil}
\rho^*(g) = \mathbb{K} \rho(a^{-1}ga) \mathbb{K} = \overline{\rho(a^{-1}ga)}. 
\end{equation}
De aqu\'i viene la notaci\'on usada en la definici\'on original de $\rho^*$. Si estamos trabajando con un espacio vectorial $V$ sin una elecci\'on de base, hay que ser m\'as cuidadoso al definir el conjugado de un operador lineal $T \colon aV \to V$. Puesto que $a \colon \mathbb{C}^n \to a\mathbb{C}^n$ se transforma en $\mathbb{K}$ bajo el isomorfismo $a\mathbb{C}^n \cong \mathbb{C}^n$, en este contexto definimos el conjugado de $T$ como sigue.
\begin{definition}
\label{operadorconjugado}
Dado un operador lineal $T \colon aV \to V$, su conjugado es el operador lineal
\begin{align*}
\overline{T} \colon V & \longrightarrow aV, \\
                    v & \longmapsto aT(av).
\end{align*}
\end{definition}

\begin{remark}
\label{homomorfismofacil}
Notemos que tenemos una partici\'on 
\begin{equation*}
    G=G_0\sqcup aG_0
\end{equation*}
del grupo $G$. Es f\'acil probar que definir un homomorfismo $f \colon G \to H$ es equivalente a proporcionar un homomorfismo $f_0 \colon G_0 \to H$ y un elemento $h \in H$ que satisfacen
\begin{enumerate}
\item $f_0(a^{-1}ga) = h^{-1}f_0(g)h$ para todo $g \in G_0$, y
\item $f_0(a^2)=h^2$.
\end{enumerate}
El homomorfismo $f$ coincidir\'ia con $f_0$ en $G_0$ y estar\'ia dado por $f(ag)=hf_0(g)$ en $aG_0$. Si $H$ es un grupo magn\'etico y $h \not\in H_0$, entonces $f$ define un morfismo de grupos magn\'eticos. Por esta raz\'on, de ahora en adelante solo damos los valores de las representaciones de un grupo magn\'etico $(G,\phi)$ en $G_0$ y en $a$, y dejamos como ejercicio para el lector comprobar que estas tres condiciones se satisfacen en cada caso.
\end{remark}

Volviendo a la situaci\'on anterior, el homomorfismo $\rho \colon G_0 \to \operatorname{GL}(V)$ da lugar a la representaci\'on
\begin{align*}
    \widetilde{\rho} \colon G& \longrightarrow \textbf{GL}(V\oplus aV),\\ g&\longmapsto \begin{cases}
        \begin{pmatrix}
            \rho(g) & 0\\ 0& \rho(g)^*
        \end{pmatrix}, & \text{ si } g\in G_0, \\ 
        & \\
        \begin{pmatrix}
            0 & \rho(a^2)a^{-1}\\ a &0
        \end{pmatrix}, & \text{ si } g=a,
    \end{cases}
\end{align*}
de $(G,\phi)$, donde $a \colon V\longrightarrow aV$ y $a^{-1} \colon aV \longrightarrow V$ son los operadores antilineales que uno se imagina. Cuando $V=\mathbb{C}^n$, entonces podemos identificar $aV$ igualmente con $\mathbb{C}^n$ y el operador $a \colon V\longrightarrow aV$ con  el operador $\mathbb{K} \colon \mathbb{C}^n\longrightarrow \mathbb{C}^n$ de conjugaci\'on compleja entrada a entrada. Con estas identificaciones, se tiene
\begin{equation}\label{induccionconcreta}
\widetilde{\rho}(a) =    \begin{pmatrix}
        0& \rho(a^2) \\ 1 &0 
    \end{pmatrix} \circ \mathbb{K}.
\end{equation}

\begin{proposition}
\label{repinducida}
    Sea $a\in G\backslash G_0$ y $V$ una representaci\'on de $G_0$. Entonces existe un isomorfismo de representaciones 
    \begin{equation*}
        \operatorname{Ind}^G_{G_0}V \longrightarrow V\oplus aV.
    \end{equation*}
\end{proposition}

\begin{proof}
Cualquier elemento de $\operatorname{Ind}^G_{G_0}V$ tiene la forma
\[ x=\sum_{i=1}^n c_ig_i\otimes v_i. \]
Sea $J_0=\{ i \in \{ 1,\ldots,n \} \mid g_i \in G_0 \}$ y $J_1$ su complemento en $\{1,\ldots,n\}$. Si $i \in J_1$, entonces podemos escribir $g_i = ag_i'$ para un \'unico $g_i' \in G_0$. Entonces tenemos
    \begin{align}
        \nonumber x & = \sum_{i=1}^n c_ig_i\otimes v_i \\ \nonumber & = \sum_{i \in J_0} c_ig_i\otimes v_i +\sum_{i \in J_1} c_ig_i\otimes v_i \\ \nonumber & = \sum_{i \in J_0} c_ie\otimes g_i\cdot v_i +\sum_{i \in J_1} c_iag'_i\otimes v_i \\ \label{inddescom} & = \sum_{i \in J_0} c_ie\otimes g_i\cdot v_i +\sum_{i \in J_1} c_ia\otimes g'_i\cdot v_i.
    \end{align}
    Con la descomposici\'on anterior podemos definir las transformaciones lineales
    \begin{align*}
        \Psi \colon \operatorname{Ind}^G_{G_0}V &\longrightarrow V\oplus aV, \\ \sum_i c_ig_i\otimes v_i & \longmapsto \left(\sum_{i \in J_0} c_i (g_i\cdot v_i)\,,\, \sum_{i \in J_1} c_ia (g'_i\cdot v_i) \right), \\
        \\
        \Phi \colon V\oplus aV & \longrightarrow \operatorname{Ind}^G_{G_0} V, \\ (v,aw) &\longmapsto e\otimes v+ a\otimes w.
    \end{align*}
    Gracias a la descomposici\'on \ref{inddescom} vemos que $\Psi$ es la inversa de $\Phi$. Dejamos como ejercicio al entusiasta lector el verificar que estas funciones son compatibles con la acci\'on de $(G,\phi)$.
\end{proof}
\begin{remark}
    Notemos que la representaci\'on conjugada $\rho^*$ depende en principio del elemento $a\in G\backslash G_0$ que se eligi\'o. Sin embargo, no es dif\'icil comprobar que diferentes elecciones de $a$ producen representaciones isomorfas.
\end{remark}

Ahora veremos que los funtores de inducci\'on y restricci\'on son adjuntos. El lector no versado en el lenguaje de categor\'ias puede ignorar la terminolog\'ia de funtores, adjuntos y naturalidad en esta proposici\'on ya que en el resto del art\'iculo solo utilizamos que existe una biyecci\'on.

\begin{proposition}[Reciprocidad de Frobenius]\label{reciprocidad}
    Los funtores $\operatorname{Res}^{G}_{G_0}$ y $\operatorname{Ind}^{G}_{G_0}$ son adjuntos, es decir, existe una biyecci\'on natural
    \begin{align*}
        \operatorname{Hom}_{(G,\phi)}(\operatorname{Ind}^G_{G_0} V,W) \longrightarrow \operatorname{Hom}_{G_0}( V, \operatorname{Res}^{G}_{G_0} W)
    \end{align*}
    para cualesquiera representaciones $V$ y $W$ de $G_0$ y $(G,\phi)$, respectivamente.
\end{proposition}
\begin{proof}
    Definamos las funciones
    \begin{align*}
        \Phi \colon \operatorname{Hom}_{(G,\phi)}(\operatorname{Ind}^G_{G_0} V,W)& \longrightarrow \operatorname{Hom}_{G_0}( V, \operatorname{Res}_{G_0}^G W) \\ f & \longmapsto \Phi(f),
    \end{align*}
    donde $\Phi(f)(v)=f(1\otimes v)$, y
    \begin{align*}
      \Psi \colon \operatorname{Hom}_{G_0}( V, \operatorname{Res}_{G_0}^G W) & \longrightarrow \operatorname{Hom}_{(G,\phi)}(\operatorname{Ind}^G_{G_0} V,W), \\ f & \longmapsto \Psi(f),
    \end{align*}
    donde $\Psi(f)(\sum_i c_ig_i\otimes v_i)=\sum_i c_ig_i\cdot f(v_i)$. Es sencillo comprobar que $\Phi$ es la inversa de $\Psi$ y que define una transformaci\'on natural.
\end{proof}

\section{Lema de Schur}

Las representaciones de grupos magn\'eticos se pueden descomponer en representaciones m\'as simples (irreducibles), al igual que las representaciones de grupos usuales. La demostraci\'on de este hecho es totalmente an\'aloga a la de representaciones de grupos por lo que solo mencionaremos los pasos a seguir. \newline

\begin{itemize}
    \item Definici\'on de representaci\'on irreducible. 
    \end{itemize}
    
    \begin{definition}
        Una representaci\'on $V$ de $(G,\phi)$ es \textbf{irreducible} si sus \'unicas subrepresentaciones son 0 y $V$.
    \end{definition}
    \begin{itemize}

    \item Descomposici\'on de representaciones. 
    \end{itemize}
    \begin{theorem}[Maschke]
        Sea $V$ una representaci\'on de $(G,\phi)$ y $V_1\subset V$ una subrepresentaci\'on. Entonces existe otra subrepresentaci\'on $V_2\subset V$ tal que $V\cong V_1\oplus V_2$ como representaciones de $(G,\phi)$.
    \end{theorem}

    \begin{itemize}
    \item Semisimplicidad. 
    \end{itemize}
    \begin{corollary}\label{descomenirreps}
        Toda representaci\'on de $(G,\phi)$ es suma directa de representaciones irreducibles.
    \end{corollary}

Recordemos el lema de Schur cl\'asico y su demostraci\'on pues ser\'a de utilidad en los siguientes resultados

\begin{lemma}[Lema de Schur]\label{lemadeSchurclasico}
    Sean $V,W$ dos representaciones irreducibles de $G_0$.
    \begin{enumerate}
        \item Si $T\in \operatorname{Hom}_{G_0}(V,W)$, entonces $T=0$ \'o $T$ es un isomorfismo.
        \item Si $T\in \operatorname{End}_{G_0}(V)$, entonces $T=\lambda\cdot \mathrm{id}_V$ con $\lambda\in \mathbb{C}$. Por lo tanto,
        \begin{equation*}
            \operatorname{End}_{G_0}(V)\cong \mathbb{C}.
        \end{equation*}
    \end{enumerate}
\end{lemma}

\begin{proof}
    Sea $T \in \operatorname{Hom}_{G_0}(V,W)$. Entonces $\operatorname{Ker}(T)$ es una subrepresentaci\'on de $V$ y $\operatorname{Im}(T)$ es una subrepresentaci\'on de $W$. Como $V$ y $W$ son irreducibles, entonces $\operatorname{Ker}(T)=V$ y $\operatorname{Im}(T)=0$ \'o $\operatorname{Ker}(T)=0$ y $\operatorname{Im}(T)=W$, es decir, $T=0$ \'o $T$ es un isomorfismo.

Sea $T \in \operatorname{End}_{G_0}(V)$. Puesto que su polinomio caracter\'istico es un polinomio con coeficientes complejos, obtenemos que $T$ tiene al menos un valor propio $\lambda \in \mathbb{C}$. Notemos que $T-\lambda\cdot \mathrm{id}_V$ es un endomorfismo de $V$ que no puede ser un isomorfismo, por lo tanto $T-\lambda\cdot \mathrm{id}_V=0$.
\end{proof}

Notemos que dada una representaci\'on irreducible $\rho \colon G_0\longrightarrow \operatorname{GL}(n)$, la representaci\'on conjugada $\rho^*$ puede ser isomorfa a $\rho$ o no. Cuando $\rho$ es isomorfa a $\rho^*$, el siguiente lema muestra que le podemos asociar un signo lo cual rompe este caso en dos posibilidades. Esto define una partici\'on dentro del conjunto de representaciones irreducibles de $G_0$ en representaciones de tipo real, complejo y cuaterni\'onico.

\begin{lemma}\label{isomorfismos}
    Sea $\rho \colon G_0 \longrightarrow \operatorname{GL}(V)$ una representaci\'on irreducible de $G_0$, $a\in G\backslash G_0$ y $\rho^* \colon G_0 \longrightarrow \operatorname{GL}(aV)$ su representaci\'on conjugada. Entonces ocurre una y solo una de las siguientes dos posibilidades.
    \begin{enumerate}
        \item Las representaciones $\rho$ y $\rho^*$ no son isomorfas.
        \item Existe un isomorfismo de representaciones $T\in \operatorname{Hom}_{G_0}(\rho^*,\rho)$ tal que 
        \begin{equation*}
            T\overline{T}=\pm \rho(a^2).
        \end{equation*} 
    \end{enumerate}
\end{lemma}

\begin{proof}
Haremos solo el caso $V=\mathbb{C}^n$. El caso general se deduce de este caso escogiendo una base para $V$ y
considerando la definici\'on general de operador conjugado, es decir, la definici\'on \ref{operadorconjugado}.

    Supongamos que existe un isomorfismo $S\in \operatorname{Hom}_{G_0}(\rho^*,\rho)$. Entonces 
    \begin{equation}\label{isomorfismodeirrep}
        \rho^*(g)=S^{-1}\rho(g)S
    \end{equation}
    para todo $g \in G_0$. Usando esta relaci\'on y la ecuaci\'on \ref{conjugadafacil}, obtenemos
    \begin{align}
    \overline{S}\rho(g) & = \mathbb{K}S\mathbb{K}\rho(a^{-1}aga^{-1}a) \nonumber\\
     & = \mathbb{K} S \rho^*(aga^{-1}) \mathbb{K} \nonumber\\
     & = \mathbb{K} \rho(aga^{-1}) S \mathbb{K} \nonumber\\
     & = \mathbb{K} \rho(a^{-1}a^2ga^{-2}a) S \mathbb{K} \nonumber\\
     & = \rho^*(a^2ga^{-2}) \overline{S}.  \label{conjugacion}
    \end{align}
    Veamos que la composici\'on $S\overline{S}\rho(a^{-2})$ es un isomorfismo de la representaci\'on $\rho$:
    \begin{align*}
        S\overline{S}\rho(a^{-2})\circ \rho(g)
        &= S\overline{S}\rho(a^{-2}g) \\
        &= S\rho^*(ga^{-2})\overline{S} \\
        &= \rho(ga^{-2})S\overline{S} \\
        &= \rho(g)\circ\rho(a^{-2})S\overline{S} \\
        &= \rho(g)\circ S\rho^*(a^{-2})\overline{S}  \\
        &= \rho(g) \circ S\overline{S}\rho(a^{-2}),
    \end{align*}
    donde usamos repetidamente las ecuaciones \ref{isomorfismodeirrep} y \ref{conjugacion}. Por el lema \ref{lemadeSchurclasico}, existe $c\in \mathbb{C}\backslash\{0\}$ tal que 
    \begin{equation}\label{morfismoclave}
        S\overline{S}=c \cdot\rho(a^2).
    \end{equation}
    Mostremos que $c$ es un n\'umero real. Se cumple
    \begin{align*} 
        \overline{c}\cdot \mathrm{id} &= \mathbb{K} (c \cdot \mathrm{id}) \mathbb{K} \\
        & = \mathbb{K}S\overline{S}\rho(a^{-2}) \mathbb{K} \\
        &= \overline{S}S\mathbb{K}\rho(a^{-2})\mathbb{K} \\
        &= \overline{S}S \rho^*(a^{-2}) \\
        &= \overline{S}\rho(a^{-2}) S \\
             &= c S^{-1}S 
             \\ &= c\cdot \mathrm{id},
    \end{align*}
    donde en el pen\'ultimo paso usamos la ecuaci\'on \ref{morfismoclave}. Por lo tanto $c\in \mathbb{R}\backslash \{0\}$. Finalmente definimos el morfismo
    \[
       T=\begin{cases}
            \displaystyle \frac{1}{\sqrt{c}}S, & \text{ si } c>0,\\
            & \\ 
            \displaystyle \frac{1}{\sqrt{-c}}S, & \text{ si } c<0, 
        \end{cases}
        \]
    de $\rho$ a $\rho^*$, que satisface la condici\'on deseada.
\end{proof}

Si $T$ es un isomorfismo como en el segundo caso del lema anterior con $T\overline{T} = \pm \rho(a^2)$, notemos que conjugar la ecuaci\'on \ref{induccionconcreta} por la matriz
\begin{equation*}
    \begin{pmatrix}
        \mathrm{id} &0 \\ 0 & T
    \end{pmatrix}
\end{equation*}
la transforma en
\begin{equation}
\label{trascambiobase}
\widehat{\rho}(a) =    \begin{pmatrix}
        0 & \pm T \\ T & 0
    \end{pmatrix} \circ \mathbb{K}.
\end{equation}
donde $\widehat{\rho}$ es la representaci\'on resultante de conjugar $\widetilde{\rho}$ por dicha matriz, y el signo de $\pm T$ coincide con el de $\pm \rho(a^2)$.

El siguiente resultado es el lema de Schur en el contexto de grupos magn\'eticos. Notemos que ahora los endomorfismos de representaciones irreducibles pueden ser m\'as que m\'ultiplos complejos de la identidad.

\begin{lemma}[Lema de Schur]\label{lemadeschur}
    Sean $V,W$ dos representaciones irreducibles de $(G,\phi)$.
    \begin{enumerate}
        \item Si $T\in \operatorname{Hom}_{(G,\phi)}(V,W)$, entonces $T$ es cero o un isomorfismo.
        \item Si $V=W$ y denotamos $V_0=\operatorname{Res}_{G_0}^G(V)$, entonces
        \[\operatorname{End}_{(G,\phi)}(V) \cong \begin{cases}
            \mathbb{R}, & \text{si } V_0 \text{ es irreducible,} \\
            \mathbb{C}, & \text{si } V_0 \text{ no es irreducible y } V_0\not\cong aV_0, \\ \mathbb{H}, & \text{si } V_0 \text{ no es irreducible y } V_0\cong aV_0. 
        \end{cases}\]
    \end{enumerate}
\end{lemma}

\begin{proof}
    El primer inciso se prueba exactamente igual que en el lema \ref{lemadeSchurclasico}. Supongamos que $\operatorname{Res}_{G_0}^G V$ es irreducible y notemos que
    \begin{align*}
        \operatorname{End}_{(G,\phi)}(V) & =  \operatorname{Hom}_{(G,\phi)}(V,V) \\ & \subset \operatorname{Hom}_{G_0}(\operatorname{Res}_{G_0}^G V,\operatorname{Res}_{G_0}^G V) \\ & \cong \mathbb{C},
    \end{align*}
    donde el \'ultimo isomorfismo se tiene por el lema \ref{lemadeSchurclasico}. Es decir, si $T\in \operatorname{End}_{(G,\phi)}(V)$, entonces $T=z\cdot \mathrm{id}$ para cierto $z \in \mathbb{C}$. Cuando consideramos la acci\'on del elemento antilineal $a$, obtenemos
    \[ z\rho(a) = T \rho(a) = \rho(a) T = \rho(a) \circ (z \cdot \mathrm{id}) = \overline{z} \rho(a), \]
    de donde $z\in \mathbb{R}$ y por lo tanto $\operatorname{End}_{(G,\phi)}(V)\cong \mathbb{R}$.
    
    Ahora supongamos que $\operatorname{Res}_{G_0}^G V$ no es irreducible. Sea $U$ una subrepresentaci\'on irreducible de $\operatorname{Res}^{G}_{G_0}V$. Entonces $U\oplus aU\subset V$ es una representaci\'on de $(G,\phi)$, y como $V$ es irreducible, tenemos que
    \begin{equation*}
        U\oplus aU= V.
    \end{equation*}
    Tenemos dos opciones. Si $U\not \cong aU$,  podemos calcular el \'algebra de endomorfismos usando la reciprocidad de Frobenius y el lema de Schur cl\'asico.
    \begin{align*}
        \operatorname{End}_{(G,\phi)}(V)&= \operatorname{Hom}_{(G,\phi)}(U\oplus aU,U\oplus aU) \\ & \cong \operatorname{Hom}_{G_0}(U,\operatorname{Res}_{G_0}^G (U\oplus aU)) \\ & \cong \operatorname{Hom}_{G_0}(U,U) \times \operatorname{Hom}_{G_0}(U,aU) \\ & \cong \mathbb{C} \times \{0\} \\ & \cong \mathbb{C}.  
    \end{align*}
    Por otro lado, si $U\cong aU$, por la proposici\'on \ref{isomorfismos} existe un isomorfismo $T\in \operatorname{Hom}_{G_0}(aU,U)$ tal que $T\overline{T}=\pm\rho(a^2)$. Sin embargo, si $T\overline{T}=\rho(a^2)$, entonces el subespacio vectorial
    \[ W = \{ (u,a^{-1}T(au)) \in V \mid u \in U \} \]
    de $V$ ser\'ia una subrepresentaci\'on propia no trivial, lo que contradice que $V$ sea irreducible. Por lo tanto, se debe tener $T\overline{T}=-\rho(a^2)$. Tras escoger una base de $U$, cualquier elemento de $\operatorname{End}_{(G,\phi)}(V)$ debe tener la forma
    \[ A=\begin{pmatrix}
        x\, \mathrm{id}& y \, \mathrm{id}\\ w \, \mathrm{id}&z\,\mathrm{id} 
    \end{pmatrix} \]
    por el lema de Schur cl\'asico, y debe conmutar con la acci\'on de $a$. Pasando a una representaci\'on isomorfa, podemos usar la ecuaci\'on \ref{trascambiobase}, de donde
    \[ A \rho(a) = \begin{pmatrix}
            x&y\\w&z
        \end{pmatrix} \begin{pmatrix}
            0 &-T\\T&0
        \end{pmatrix}\circ \mathbb{K} = \begin{pmatrix}
            yT&-xT\\zT& -wT
        \end{pmatrix}\circ\mathbb{K}, \]
    \[ \rho(a)A = \begin{pmatrix}
            0 &-T\\T&0
        \end{pmatrix}\circ \mathbb{K} \begin{pmatrix}
            x&y\\w&z
        \end{pmatrix} = \begin{pmatrix}
            -\overline{w}T&-\overline{z}T\\ \overline{x}T& \overline{y}T
        \end{pmatrix}\circ\mathbb{K}. \]
    Entonces
    \begin{equation*}
        y=-\overline{w}, \hspace{1cm} x=\overline{z},
    \end{equation*}
    y por lo tanto 
    \begin{equation*}
        \operatorname{End}_{(G,\phi)}(V)=\left\{\begin{pmatrix}x&y\\-\overline{y} & \overline{x}\end{pmatrix}\, \Bigg| \, x,y\in\mathbb{C}\right\}\cong \mathbb{H}. \qedhere
    \end{equation*}
\end{proof}

Diremos que una representaci\'on $V$ de $(G,\phi)$ es de tipo real, compleja o cuaterni\'onica si $\operatorname{End}_{(G,\phi)}(V)\cong\mathbb{R}^n$, $\mathbb{C}^n$ \'o $\mathbb{H}^n$, respectivamente para alg\'un $n\in \mathbb{N}$. 

\begin{corollary}\label{criteriodeirrep}
    Sea $V$ una representaci\'on de $(G,\phi)$. Entonces $V$ es una representaci\'on irreducible de tipo real, complejo o cuaterni\'onico si y solo si $\operatorname{End}_{(G,\phi)}(V) \cong \mathbb{R}$, $\mathbb{C}$ \'o $\mathbb{H}$, respectivamente. 
\end{corollary}
\begin{proof}
    El lema de Schur para grupos magn\'eticos (lema \ref{lemadeschur}) nos garantiza la necesidad. Para la suficiencia tomemos una representaci\'on $V$ tal que $\operatorname{End}_{(G,\phi)}(V) \cong \mathbb{R}$, $\mathbb{C}$ \'o $\mathbb{H}$. Por el corolario \ref{descomenirreps}, podemos descomponer $V$ como suma de representaciones irreducibles y de nuevo por el lema \ref{lemadeschur}, esta suma se divide en representaciones irreducibles de tipo real, complejo o cuaterni\'onico, es decir,
    \begin{equation*}
        V\cong \left(\oplus_i A_i^{\oplus\alpha_i}\right) \oplus \left(\oplus_j B_j^{\oplus\beta_j}\right)\oplus \left(\oplus_k C_k^{\oplus\gamma_k}\right), 
    \end{equation*}
    con $A_i$, $B_j$, $C_k$ representaciones irreducibles no isomorfas de tipo real, complejo o cuaterni\'onico, respectivamente. Tomando endomorfismos de la descomposici\'on anterior, obtenemos
    \begin{equation*}
        \operatorname{End}_{(G,\phi)}(V) \cong \mathbb{R}^{\alpha} \times\mathbb{C}^{\beta} \times \mathbb{H}^{\gamma},
    \end{equation*}
    donde $\alpha$ es la suma de los $\alpha_i$ y similarmente para $\beta$ y $\gamma$. Por lo tanto debe existir un \'unico \'indice $i$, $j$ \'o $k$ tal que alguna de las multiplicidades $\alpha_i$, $\beta_j$ \'o $\gamma_k$ es igual a uno y todas las dem\'as multiplicidades son cero.
\end{proof}

\section{Clasificaci\'on de representaciones irreducibles}

Cada uno de los tres diferentes tipos de representaciones irreducibles de $G_0$ da lugar a una representaci\'on irreducible de $(G,\phi)$, y viceversa, toda representaci\'on irreducible de $(G,\phi)$ proviene de una representaci\'on irreducible de $G_0$. La primera demostraci\'on de este hecho se debe a Wigner y se puede encontrar en \cite{wigner}. 

Diremos que una representaci\'on irreducible $\rho$ de $G_0$ es de tipo compleja si $\rho$ no es isomorfa a $\rho^*$. 
Diremos que es de tipo real si existe un isomorfismo $T \colon \rho^* \to \rho$ tal que $T\overline{T}=\rho(a^2)$, y de tipo cuaterni\'onica si existe un isomorfismo $T \colon \rho^* \to \rho$ tal que $T\overline{T}=-\rho(a^2)$. Por el lema \ref{isomorfismos}, cualquier representaci\'on irreducible de $G_0$ debe cumplir una de estas tres alternativas.

\begin{theorem}[Clasificaci\'on de Wigner]\label{claswigner}
    Sean $(G,\phi)$ un grupo magn\'etico, $a\in G\backslash G_0$ y $\rho \colon G_0 \longrightarrow \operatorname{GL}(n)$ una representaci\'on irreducible. Denotemos por $\rho^*$ a su representaci\'on conjugada.

    \begin{itemize}
        \item Si $\rho$ es de tipo real y \begin{equation*}
            T \colon \rho^*\longrightarrow \rho
        \end{equation*}
        es un isomorfismo de representaciones tal que $T\overline{T}=\rho(a^2)$, entonces produce la representaci\'on irreducible
        \begin{equation}\label{realrep}
            g\cdot v = \begin{cases}
                \rho(g)v, & \text{ si } g\in G_0,
                \\ T\circ\mathbb{K}(v), & \text{ si } g=a,
            \end{cases}
        \end{equation}
        de $(G,\phi)$ sobre $\mathbb{C}^n$.
        \item Si $\rho$ es de tipo complejo, entonces $\operatorname{Ind}^{G}_{G_0}\rho$ es una representaci\'on irreducible de $(G,\phi)$.
        \item Si $\rho$ es de tipo cuaterni\'onico y 
        \begin{equation*}
            T \colon \rho^*\longrightarrow \rho
        \end{equation*}
        es un isomorfismo de representaciones 
        tal que $T\overline{T}=-\rho(a^2)$, entonces $\operatorname{Ind}^{G}_{G_0}\rho$ es una representaci\'on irreducible de $(G,\phi)$, que tras el cambio de base por la matrix $\begin{pmatrix}
            \mathrm{id} & 0\\ 0 & T
        \end{pmatrix}$ tiene la forma
        \begin{align*}
        G & \longrightarrow \textbf{GL}(2n) \\
        g & \mapsto \begin{cases}
                \begin{pmatrix}
                \rho(g) &0 \\ 0 & \rho(g)
            \end{pmatrix} , & \text{ si } g\in G_0,
                \\
                & \\ \begin{pmatrix} 0 & -T \\ T & 0
            \end{pmatrix} \circ \mathbb{K}, & \text{ si } g=a.
            \end{cases}
        \end{align*} 
    \end{itemize}
    M\'as a\'un, toda representaci\'on irreducible de $(G,\phi)$ es de una de estas formas. 
\end{theorem}

\begin{proof}
    Supongamos que $\rho$ y $\rho^*$ no son isomorfas. Calculemos el \'algebra de endomorfismos de la representaci\'on $\operatorname{Ind}^{G}_{G_0}\rho$ usando la reciprocidad de Frobenius. Se tiene
    \begin{align*}
        \operatorname{End}_{(G,\phi)}(\operatorname{Ind}^{G}_{G_0}\rho) & =   \operatorname{Hom}_{(G,\phi)}(\operatorname{Ind}^{G}_{G_0}\rho,\operatorname{Ind}^{G}_{G_0}\rho) \\ & \cong \operatorname{Hom}_{G_0}(\rho, \operatorname{Res}_{G_0}^G \operatorname{Ind}^{G}_{G_0}\rho) \\ & \cong \operatorname{Hom}_{G_0}(\rho, \rho\oplus \rho^*)  \\ & \cong \operatorname{Hom}_{G_0}(\rho, \rho)\times \operatorname{Hom}_{G_0}(\rho, \rho^*) \\ & \cong \mathbb{C}, 
    \end{align*}
    donde el antepen\'ultimo isomorfismo se tiene por la proposici\'on \ref{repinducida}. Por el corolario \ref{criteriodeirrep}, concluimos que $\operatorname{Ind}^G_{G_0} \rho$ es irreducible de tipo complejo.
    
    Si existe un isomorfismo $T \colon \rho^*\longrightarrow \rho$ tal que $T\overline{T}=\rho(a^2)$, entonces la ecuaci\'on \ref{realrep} define una representaci\'on de $(G,\phi)$ por la observaci\'on \ref{homomorfismofacil}. M\'as a\'un, en esta ecuaci\'on vemos que su restricci\'on a $G_0$ es $\rho$, que es irreducible, as\'i que define una representaci\'on irreducible de $(G,\phi)$. Por el lema \ref{lemadeschur}, es de tipo real.
    
    Si existe un isomorfismo $T \colon \rho^* \longrightarrow \rho$ tal que $T\overline{T}=-\rho(a^2)$, la \'ultima parte de la demostraci\'on del lema de Schur (lema \ref{lemadeschur}), muestra que $\operatorname{End}_{(G,\phi)}(\operatorname{Ind}^{G}_{G_0}\rho) \cong \mathbb{H}$. Por el corolario \ref{criteriodeirrep}, es una representaci\'on irreducible de $(G,\phi)$ de tipo cuaterni\'onico. La forma expl\'icita en $g=a$ tambi\'en se vio en esa demostraci\'on, mientras que la forma expl\'icita en $G_0$ se debe a que 
    \[ \operatorname{Res}_{G_0}^G \operatorname{Ind}_{G_0}^G \rho \cong \rho \oplus \rho^* \cong \rho \oplus \rho, \]
    donde el primer isomorfismo se tiene por la Proposici\'on \ref{repinducida} y el segundo porque $\rho \cong \rho^*$.

    Finalmente, si $\alpha$ es una representaci\'on irreducible de $(G,\phi)$, entonces la demostraci\'on del lema de Schur para grupos magn\'eticos (lema \ref{lemadeschur}) muestra que $\alpha$ se construye a partir de las representaciones irreducibles en $\operatorname{Res}_{G_0}^G \alpha$. 
\end{proof}

El teorema de clasificaci\'on de Wigner define un funtor
\begin{align*}
    W_G \colon \operatorname{Rep}(G_0)& \longrightarrow \textbf{Rep}(G,\phi), \\ \rho & \longmapsto W_G(\rho),
\end{align*}
que manda representaciones irreducibles de $G_0$ a representaciones irreducibles de $(G,\phi)$.

\section{Ejemplos}

En esta secci\'on determinaremos las representaciones irreducibles de los grupos magn\'eticos $(\mathbb{Z}/2n,\operatorname{mod}2)$, $(D_n=\mathbb{Z}/n\rtimes\mathbb{Z}/2, \pi_2)$, $(Q_8, Q_8\overset{p}{\longrightarrow}Q_8/\langle i\rangle)$.

\subsection{ \texorpdfstring{$(\mathbb{Z}/2n,\operatorname{mod}2)$}{TEXT} }
Notemos que en este caso solo consideramos grupos c\'iclicos de orden par puesto que los de orden impar solo tienen homomorfismos triviales a $\mathbb{Z}/2$. El grupo magn\'etico $(\mathbb{Z}/2,\operatorname{mod}2)$ genera la siguiente sucesi\'on exacta corta:
\begin{equation*}
    \xymatrix{0 \ar[r] & G_0=\mathbb{Z}/n \ar[r]^{\times 2} & G=\mathbb{Z}/2n \ar[rr]^{\hspace{0.5cm}\phi=\operatorname{mod}2} && \mathbb{Z}/2 \ar[r] &0. }
\end{equation*}
Elegimos $a=1\in \mathbb{Z}/2n$. El grupo $G_0=\mathbb{Z}/n$ tiene $n$ representaciones irreducibles dadas por
\begin{align*}
    \chi_{l} \colon \mathbb{Z}/n& \longrightarrow \operatorname{GL}(1) \cong \mathbb{C}^{\times}, \\ 1 & \longmapsto \exp \frac{2\pi i l}{n}, 
\end{align*}
para $l=0,\ldots,n-1$. Usaremos la notaci\'on $\exp x$ para $e^x$ porque estaremos manipulando exponentes fraccionarios. Aqu\'i debemos distinguir dos casos:

\begin{itemize}
    \item Si $n$ es par, aparecen los tres tipos de representaciones. Para ver esto calculemos la representaci\'on asociada en los casos $l=0$, $l=n/2$ y $l\not= 0,n/2$.
    
    Cuando $l=0$, la representaci\'on conjugada est\'a dada por 
    \[ \chi^*_0(g) =  \overline{\chi}_0(g) = 1 =  \chi_0(g), \]
        as\'i que $\chi_0^*$ es isomorfa a $\chi_0$ con isomorfismo $T=1$. M\'as a\'un, $T\overline{T}=\chi_0(a^2)$, por lo que $\chi_0$ es de tipo real y genera la representaci\'on irreducible 
        \begin{align*}
        W_{\mathbb{Z}/2n}(\chi_0) \colon (\mathbb{Z}/2n,\operatorname{mod}2) & \longrightarrow \textbf{GL}(\mathbb{C}), \\ l & \longmapsto \mathbb{K}^{l}.
        \end{align*}
        de $(\mathbb{Z}/2n,\operatorname{mod}2)$.
        
    Si $l=n/2$, calculemos la representaci\'on conjugada 
    \[ \chi^*_{n/2}(g) =  \overline{\left( \exp \frac{2\pi i n/2}{n} \right)^g} =  (-1)^g =  \chi_{n/2}(g). \]
        Por lo tanto $\chi^*_{n/2}$ y $\chi_{n/2}$ son isomorfas con isomorfismo $T=\mathrm{id}$. Adem\'as se tiene que 
        \begin{equation*}
            T\overline{T}=-\chi_{n/2}(a^2),
        \end{equation*}
        por lo que $\chi_{n/2}$ es de tipo cuaterni\'onico. Entonces $\chi_{n/2}$ genera la representaci\'on irreducible de $(\mathbb{Z}/2n,\operatorname{mod}2)$
        dada por
        \begin{align*}
            W_{\mathbb{Z}/2n}(\chi_{n/2}) \colon (\mathbb{Z}/2n,\operatorname{mod}2) & \longrightarrow \textbf{GL}(\mathbb{C}^2), \\ 1 & \longmapsto \begin{pmatrix}
                0 & -1 \\ 1& 0 
            \end{pmatrix}\circ \mathbb{K}.
        \end{align*}
    
    Finalmente, si $l\not = 0, n/2$, calculemos la representaci\'on conjugada
    \[ \chi^*_l(g)= \overline{\left(\exp{\frac{2\pi i l}{n}}\right)^g} = \left(\exp{\frac{2\pi i(n-l)}{n}}\right)^g = \chi_{n-l}(g). \]
        Como $\chi_l$ y $\chi_{n-l}$ no son isomorfas, entonces $\chi_l$ genera la siguiente representaci\'on irreducible 
        \begin{align*}
            W_{\mathbb{Z}/2n}(\chi_l) \colon (\mathbb{Z}/2n\operatorname{mod}2), & \longrightarrow \textbf{GL}(\mathbb{C}^2), \\ 1 & \longmapsto \begin{pmatrix}
                0& \exp{\frac{2\pi i}{n}} \\ 1 & 0
            \end{pmatrix}\circ\mathbb{K}.
        \end{align*}
de $(\mathbb{Z}/2n,\operatorname{mod}2)$. \newline
    \item Si $n$ es impar, tenemos solo el primer y tercer tipo de representaciones que aparecen en el caso anterior.
\end{itemize}

\subsection{\texorpdfstring{$(D_n=\mathbb{Z}/n\rtimes\mathbb{Z}/2, \pi_2)$}{TEXT}}

En la descripci\'on de $D_n$ como producto semidirecto, el subgrupo $\mathbb{Z}/n$ es el generado por las rotaciones y $\mathbb{Z}/2$ el generado por una reflexi\'on fija. La acci\'on del elemento no trivial $1\in\mathbb{Z}/2$ en $\mathbb{Z}/n$ est\'a dada por
\begin{equation*}
    1\cdot l = -l=n-l.
\end{equation*}
Este grupo magn\'etico produce la sucesi\'on exacta corta
    \[ 1 \longrightarrow G_0=\mathbb{Z}/n \longrightarrow G  = \mathbb{Z}/n\rtimes \mathbb{Z}/2 \stackrel{\pi_2}{\longrightarrow} \mathbb{Z}/2 \longrightarrow 1 \]
Tomemos $a=(0,1)$. En este caso, $G_0=\mathbb{Z}/n$ es el mismo grupo que en el ejemplo anterior. Calculemos las representaciones asociadas.
\begin{align*}
    \chi^*_l(g) & = \overline{\chi}_l(a^{-1}ga) \\ & = \overline{\left(\exp{\frac{2\pi i l}{n}} \right)^{n-g}} \\ & =  \exp{\frac{2\pi il(g-n)}{n}} \\ & =  \left(\exp{\frac{2\pi il}{n}}\right)^g \\ & =  \chi_l(g).
\end{align*}
Puesto que $\chi_l\cong \chi^*_l$ y el isomorfismo es $T=\mathrm{id}$, tenemos
\begin{equation*}
    T\overline{T}=1=\chi_l(0)=\chi_l(a^2).    
\end{equation*}
Por lo tanto todas las representaciones $\chi_l$ son de tipo real y generan las $n$ representaciones irreducibles de $(D_n,\pi_2)$ dadas por
\begin{align*}
    W_{D_n}(\chi_l) \colon (D_n,\pi_2) & \longrightarrow \textbf{GL}(\mathbb{C}), \\ (n,\varepsilon)& \longmapsto \chi_l(n)\circ\mathbb{K}^{\varepsilon},
\end{align*}
para $0 \leq l \leq n-1$.

\subsection{\texorpdfstring{$(Q_8,\overset{\hspace{0.5cm}q}{Q_8\longrightarrow} Q_8/\langle i\rangle)$}{TEXT}}

El grupo de los cuaterniones $Q_8=\{1,-1,i,-i,j,-j,k,-k\}$ contiene al subgrupo normal $G_0=\langle i\rangle\cong \mathbb{Z}/4$. El cociente de $Q_8$ por $G_0$ tiene orden dos, por lo cual podemos considerar al grupo magn\'etico $(Q_8,q)$, donde $q$ es la composici\'on del cociente $Q_8 \to Q_8/\langle i \rangle$ con el isomorfismo $Q_8/\langle i \rangle \to \mathbb{Z}/2$. Esto determina la sucesi\'on exacta
    \[ 1 \longrightarrow G_0=\langle i \rangle \longrightarrow G  = Q_8 \stackrel{q}{\longrightarrow} \mathbb{Z}/2 \longrightarrow 1. \]
Tomemos $a=j$ y calculemos las representaciones conjugadas.
\begin{align*}
    \chi^*_l(g) & =  \overline{\left(\exp{\frac{2\pi i l}{4}}\right)^{-jgj}} \\ & =  \overline{\left(i^l\right)^{-jgj}} \\ & =  \begin{cases}
       1 & \text{ si } l=0\\ \overline{i^{g^{-1}}} & \text{ si }  l=1\\ (-1)^{g^{-1}} & \text{ si }  l=2\\ \overline{(-i)^{g^{-1}}} & \text{ si }  l=3
    \end{cases} \\ & =  \begin{cases}
       1 & \text{ si }  l=0\\ i^{g} & \text{ si }  l=1\\ (-1)^{g} & \text{ si }  l=2\\ (-i)^{g} & \text{ si }  l=3
    \end{cases} \\ & = \chi_l(g).
\end{align*}
Por lo tanto la representaci\'on $\chi^*_l$ es isomorfa a $\chi_l$ con isomorfismo $T=\mathrm{id}$. Tenemos que
\begin{equation*}
    T\overline{T}=(-1)^{l}\chi_l(a^2),
\end{equation*}
por lo que $\chi_0$ y $\chi_2$ son reales, mientras que $\chi_1$ y $\chi_3$ son cuaterni\'onicas. Las dos primeras representaciones generan las representaciones irreducibles
\begin{align*}
    W_{Q_8}(\chi_0) \colon (Q_8,q) & \longrightarrow \textbf{GL}(\mathbb{C}), \\ i & \longmapsto 1, \\ j & \longmapsto \mathbb{K},
\end{align*}

\begin{align*}
    W_{Q_8}(\chi_2) \colon (Q_8,q) & \longrightarrow \textbf{GL}(\mathbb{C}), \\ i & \longmapsto -1, \\ j & \longmapsto \mathbb{K}.
\end{align*}
Las otras dos representaciones generan las siguientes representaciones irreducibles:
\begin{align*}
    W_{Q_8}(\chi_1) \colon (Q_8,q) & \longrightarrow \textbf{GL}(\mathbb{C}^2), \\ i & \longmapsto \begin{pmatrix}
        i & 0\\ 0 & i
    \end{pmatrix}, \\ j & \longmapsto \begin{pmatrix}
        0 & -1 \\ 1 & 0
    \end{pmatrix}\mathbb{K},
\end{align*}

\begin{align*}
    W_{Q_8}(\chi_3) \colon (Q_8,q) & \longrightarrow \textbf{GL}(\mathbb{C}^2), \\ i & \longmapsto \begin{pmatrix}
        -i & 0\\ 0 & -i
    \end{pmatrix}, \\ j & \longmapsto \begin{pmatrix}
        0 & -1 \\ 1 & 0
    \end{pmatrix}\mathbb{K}.
\end{align*}

\section{Comentarios finales}

La teor\'ia de representaciones de grupos magn\'eticos que presentamos aqu\'i tiene m\'ultiples expansiones y generalizaciones. Por ejemplo, ?`por qu\'e fijarnos solamente en el grupo de Galois de la extensi\'on $\mathbb{C}/\mathbb{R}$? Si nos olvidamos de la aplicaci\'on a la f\'isica, entonces podemos considerar extensiones de Galois generales $F/k$ y trabajar con la teor\'ia de representaciones de grupos $(G, \phi \colon G\to \mathrm{Aut}_{k}(F))$ donde $G$ es un grupo y $\phi$ es un homomorfismo al grupo de automorfismos de la extensi\'on de campos. En esta direcci\'on se puede consultar \cite{Gille_Szamuely_2006}. 

Otra expansi\'on de la teor\'ia de representaciones magn\'eticas y que ha estado activa en los \'ultimos a\~nos es la teor\'ia $K$ equivariante. Hasta ahora solo hemos considerado espacios vectoriales aislados con acciones de un grupo magn\'etico, pero ?`qu\'e pasa si tenemos una familia de estos espacios parametrizada por un espacio topol\'ogico compacto Hausdorff con una acci\'on del grupo? En este contexto, los objetos con los que se trabajan son los haces vectoriales equivariantes. Dichos objetos tambi\'en son usados en materia condensada para modelar fases topol\'ogicas de la materia, el lector interesado puede consultar \cite{FreedMoore2013}, \cite{Gomi2023} y \cite{UribeSerranoMerino2025}.

Para quienes deseen ver estas ideas en acci\'on, en su entorno natural, la f\'isica de la materia condensada, pueden consultar los art\'iculos \cite{PhysRevB.110.125129}, \cite{s57q-q7gt}, \cite{PhysRevB.111.085147}  y las referencias ah\'i contenidas. Cabe se\~nalar que en dichos trabajos, los grupos magn\'eticos y sus representaciones constituyen una herramienta subyacente. Aunque no siempre aparezcan de manera expl\'icita, muchas de las construcciones te\'oricas dependen crucialmente de que los materiales considerados exhiban estas simetr\'ias.

Para finalizar y como nota hist\'orica, en los primeros a\~nos de la mec\'anica cu\'antica, la teor\'ia de grupos y sus representaciones fue recibida con escepticismo en este campo. Algunos f\'isicos, como Erwin Schr\"odinger, 
llegaron a llamarla despectivamente ``Gruppenpest'' (la ``peste'' de la teor\'ia de grupos), al considerarla demasiado formal y alejada de los problemas pr\'acticos de la f\'isica. Con el tiempo, sin embargo, qued\'o claro que 
este lenguaje matem\'atico es esencial para entender las simetr\'ias de los sistemas y la estructura de las part\'iculas elementales.

\bibliographystyle{amsplain}
\bibliography{mybibfile.bib}

\providecommand{\bysame}{\leavevmode\hbox to3em{\hrulefill}\thinspace}
\providecommand{\MR}{\relax\ifhmode\unskip\space\fi MR }
\providecommand{\MRhref}[2]{%
  \href{http://www.ams.org/mathscinet-getitem?mr=#1}{#2}
}
\providecommand{\href}[2]{#2}
\begin{thebibliography}{10}

\bibitem{bradley2010mathematical}
C.~Bradley and A.~Cracknell, \emph{The mathematical theory of symmetry in
  solids: Representation theory for point groups and space groups}, EBSCO ebook
  academic collection, OUP Oxford, 2010.

\bibitem{fedorov1971symmetry}
E.~S. Fedorov, \emph{Symmetry of crystals}, American Crystallographic
  Association Monograph, no.~7, American Crystallographic Association, Buffalo,
  N.Y., 1971, English translation of the original 1891 Russian work.

\bibitem{FreedMoore2013}
Daniel~S. Freed and Gregory~W. Moore, \emph{Twisted equivariant matter},
  Annales Henri Poincaré \textbf{14} (2013), no.~8, 1927--2023.

\bibitem{FultonHarris1991}
William Fulton and Joe Harris, \emph{Representation theory: A first course},
  Graduate Texts in Mathematics, vol. 129, Springer-Verlag, New York, NY, 1991.

\bibitem{Gallego:ks5532}
Samuel~V. Gallego, J.~Manuel Perez-Mato, Luis Elcoro, Emre~S. Tasci, Robert~M.
  Hanson, Koichi Momma, Mois~I. Aroyo, and Gotzon Madariaga, \emph{{\it
  MAGNDATA}: towards a database of magnetic structures. {I}. {T}he commensurate
  case}, Journal of Applied Crystallography \textbf{49} (2016), no.~5,
  1750--1776.

\bibitem{Gille_Szamuely_2006}
Philippe Gille and Tamás Szamuely, \emph{Central simple algebras and {G}alois
  cohomology}, Cambridge Studies in Advanced Mathematics, Cambridge University
  Press, 2006.

\bibitem{Gomi2023}
Kiyonori Gomi, \emph{Freed–{M}oore {K}-theory}, Communications in Analysis
  and Geometry \textbf{31} (2023), no.~4, 883--937.

\bibitem{PhysRevB.111.085127}
Rafael Gonz\'alez-Hern\'andez, Higinio Serrano, and Bernardo Uribe, \emph{Spin
  {C}hern number in altermagnets}, Phys. Rev. B \textbf{111} (2025), 085127.

\bibitem{PhysRevB.110.125129}
Rafael Gonz\'alez-Hern\'andez and Bernardo Uribe, \emph{Average spin {C}hern
  number}, Phys. Rev. B \textbf{110} (2024), 125129.

\bibitem{s57q-q7gt}
\bysame, \emph{Model {H}amiltonian for altermagnetic topological insulators},
  Phys. Rev. B \textbf{112} (2025), 184101.

\bibitem{RevModPhys.82.3045}
M.~Z. Hasan and C.~L. Kane, \emph{Colloquium: Topological insulators}, Rev.
  Mod. Phys. \textbf{82} (2010), 3045--3067.

\bibitem{UribeSerranoMerino2025}
Bernardo~Uribe Jongbloed, Higinio~Serrano Garcia, and Miguel
  Alejandro~Xicot{\'e}ncatl Merino, \emph{Magnetic equivariant {K}-theory:
  Topological invariants of magnetic symmetries}, 1 ed., SpringerBriefs in
  Mathematics, Springer Cham, 2025.

\bibitem{cherninsulators}
Panagiotis Kotetes, \emph{Chern insulators—fundamentals}, Topological
  Insulators, 2053-2571, Morgan \& Claypool Publishers, 2019, pp.~5--1 to
  5--22.

\bibitem{Litvin2013}
Daniel~B. Litvin, \emph{Magnetic group tables: 1-, 2- and 3-dimensional
  magnetic subperiodic groups and magnetic space groups}, International Union
  of Crystallography, 2013.

\bibitem{PhysRevB.111.085147}
Victor Mendoza-Estrada, Rafael Gonz\'alez-Hern\'andez, Bernardo Uribe, and
  Libor \ifmmode~\check{S}\else \v{S}\fi{}mejkal, \emph{Eightfold degenerate
  {D}irac nodal line in the collinear antiferromagnet
  $\mathrm{{M}n_{5}{S}i_{3}}$}, Phys. Rev. B \textbf{111} (2025), 085147.

\bibitem{ShubnikovBelov1964}
A.~V. Shubnikov and N.~V. Belov, \emph{Colored symmetry}, Pergamon Press,
  Oxford; New York, 1964, Translation of works originally published in Russian
  between 1951 and 1958.

\bibitem{wigner}
Eugene~P. Wigner, \emph{Group theory and its application to the quantum
  mechanics of atomic spectra}, Pure and Applied Physics, Vol. 5, Academic
  Press, New York-London, 1959, Expanded and improved ed. Translated from the
  German by J. J. Griffin.

\bibitem{Yang2024}
Jian Yang, Zheng-Xin Liu, and Chen Fang, \emph{Symmetry invariants and classes
  of quasiparticles in magnetically ordered systems having weak spin-orbit
  coupling}, Nature Commun. \textbf{15} (2024), no.~1, 10203.

\end{thebibliography}

\end{document}